\documentclass{amsart}
\usepackage{amssymb}
\usepackage[matrix,arrow,curve]{xy}
\usepackage{verbatim}
\usepackage{a4wide}
\usepackage[applemac]{inputenc}
\newtheorem{theorem}{Theorem}[section]
\newtheorem{lemma}[theorem]{Lemma}
\newtheorem{remark}[theorem]{Remark}

\newtheorem{corollary}[theorem]{Corollary}

\newtheorem{proposition}[theorem]{Proposition}

\def\lcm{\mathrm{lcm}}
\def\gcd{\mathrm{gcd}}
\def\diag{\mathrm{diag}}
\def\GL{\mathrm{GL}}
\def\SL{\mathrm{SL}}

\def\PSU{\mathrm{PSU}}

\def\Sp{\mathrm{Sp}}
\def\SO{\mathrm{SO}}
\def\SU{\mathrm{SU}}

\def\C{\mathbf{C}}
\def\Z{\mathbf{Z}}

\def\F{\mathbf{F}}
\def\Q{\mathbf{Q}}
\def\C{\mathbf{C}}
\def\R{\mathbf{R}}
\def\un{\mathbf{1}}

\def\Ker{\mathrm{Ker}}

\def\Aut{\mathrm{Aut}}
\def\Out{\mathrm{Out}}

\def\Ind{\mathrm{Ind}}
\def\into{\hookrightarrow}
\def\onto{\twoheadrightarrow}

\def\Id{\mathrm{Id}\,}

\def\ii{\mathrm{i}}
\def\la{\lambda}
\def\dd{\mathrm{d}}

\date{November 30, 2015.}
\usepackage{tikz}
\usetikzlibrary{backgrounds,fit}
\tikzset{every picture/.style = {inner sep=0pt,baseline}}
\tikzset{p/.style = {draw, shape          = circle,
                                 text           = black,
                                 font=\tiny,
                                 inner sep      = 0pt,
                                 outer sep      = 0pt,
                                 minimum size   = 1pt}}

\tikzset{r/.style = {red, very thin}}
\tikzset{g/.style = {green, very thin}}
\tikzset{b/.style = {blue, very thin}}
\tikzset{o/.style = {orange, very thin}}
\tikzset{u/.style = {black, very thin}}

\newcommand*\circled[1]{\tikz[baseline=(char.base)]{
            \node[shape=circle,draw,inner sep=2pt] (char) {#1};}}
            
\begin{document}
\centerline{}

\title[Crystallographic groups from reflection groups]{Crystallographic groups and flat manifolds from complex reflection groups}
\author[I.~Marin]{Ivan Marin}
\address{LAMFA, Universit\'e de Picardie-Jules Verne, Amiens, France}
\email{ivan.marin@u-picardie.fr}
\medskip

\begin{abstract} 
Following an idea of Gon\c calves, Guaschi and Ocampo on the usual braid group
we construct crystallographic and Bieberbach groups as (sub)quotients of the
generalized braid group associated to an arbitrary complex reflection group.
\end{abstract}

\maketitle

\tableofcontents

\section{Introduction}

In the paper \cite{GGO}, Gon\c calves, Guaschi and Ocampo
notice that the quotient of the (usual)  braid group $B_n$ by the commutator subgroup $(P_n,P_n)$
of the pure braid is a crystallographic group. They then prove the remarkable fact
that this quotient has no 2-torsion. From this, they can build Bieberbach groups,
namely torsion-free crystallographic groups (characterizing compact flat manifolds), by taking the
preimage under the natural projection of a 2-subgroup of the symmetric group $\mathfrak{S}_n$. In the same
paper, they prove a number of results on the finite-order elements and fnite subgroups
of this quotient $B_n/(P_n,P_n)$.

We prove here that \emph{all} the results of \cite{GGO} can be generalized to the following more
general setting, and we provide at the same time possibly simpler proofs. Let $W$ be a (finite) complex reflection group, and $B$ the corresponding
generalized braid group in the sense of \cite{BMR}. If $W$ is a real reflection group (a.k.a. finite
Coxeter group), then $B$
is an Artin group of finite Coxeter type. The quotient $B/(P,P)$ of $B$ by the commutator
subgroup of the pure braid group $P$ has been studied in the `real' case by
J. Tits (under the name `$V$') in his seminal 1966 paper \cite{TITS}  and by F. Digne in the unpublished paper \cite{DIGNE}. In the general `complex'
case, it has been studied by the author in \cite{ARREFL} and by V. Beck in \cite{BECK}.

In the present paper, we prove that this quotient is always a crystallographic group, and that it never contains
elements of order $2$. This provides a way to construct Bieberbach groups,
by taking the preimage of the 2-Sylow subgroup of the quotient $W/Z(W)$ of $W$
by its center. This is done in section \ref{sect:mainconstruction} of the paper, in which we also prove that
parabolic inclusions between reflection groups induce inclusions between the corresponding crystallographic groups.

In section \ref{sect:posstorsion}, we describe a general way to construct elements of finite
order inside $B/(P,P)$, generalizing the elements constructed in \cite{GGO}.
It is based on Springer's theory of regular elements and regular numbers. We connect
the possible orders of these elements, that we call the \emph{freely regular numbers},
with a previously introduced integer $\kappa(W)$,
which is naturally associated with the extension $1 \to P^{ab} \to B/(P,P) \to W \to 1$
(where $P^{ab} = P/(P,P)$ denotes the abelianization of $P$).
Conversely, we prove a criterion ensuring that $B/(P,P)$ cannot
contain elements of certain orders, in addition to the powers of $2$. This enables us to
prove that a few groups of the form $B/(P,P)$ actually \emph{are} Bieberbach groups.
For this study, we need to provide a detailed description of the regular elements in complex
reflection groups, for which we could not find an adequate reference.
Finally, we explain and generalize the construction given in \cite{GGO} of a non-abelian finite subgroup of order $21$ inside $B_7/(P_7,P_7)$.

In the final section \ref{sect:kahler} we explore, focusing on low-dimensional examples,
to which extent the constructed manifolds can be endowed with a Kähler structure.

As a concluding remark, we notice that \emph{none} of our proofs need to use the Shephard-Todd classification of irreducible complex reflection groups, \emph{except} the result
that freely regular numbers are necessarily coprime to $\kappa(W)$. It would be nice
to have a proof which does not use the classification for this result, too.

\medskip

{\bf Acknowledgements.} I would like to thank the anonymous
referee for a thorough reading of the article and insightful
suggestions, and V. Beck for useful remarks.

\bigskip

\section{Main construction}

\label{sect:mainconstruction}
\subsection{Reminder on Bieberbach groups}

Our reference on the subject will be \cite{DEKIMPE}. Recall from there that every compact flat manifold
can be obtained as a quotient $\R^N/\Gamma$, where $\Gamma$ is any torsion-free cocompact (a.k.a. uniform)
discrete subgroup of the group of affine isometries $\R^N \rtimes O_N(\R)$ of the Euclidean space $\R^N$.
Moreover, such groups, considered up to isomorphisms, completely characterize the manifold. This is
the content of the famous Bieberbach theorems, and these groups are called Bieberbach groups. Removing
the `torsion-free' assumption defines the larger class of so-called crystallographic groups.

A classical result (see \cite{DEKIMPE}, theorem 2.1.4) states that an abstract group $\Gamma$
is crystallographic if and only if there exists a short exact sequence
$$
1 \to \Z^N \to \Gamma \to \Phi \to 1
$$
where $\Phi$ is finite group and $\Z^N$ is maximal abelian inside $\Gamma$. This is clearly equivalent to
saying that $\Phi$ is finite and the natural action $\Phi \to \Out(\Z^N)=\Aut(\Z^N) = \GL_N(\Z)$
is faithful. If $\Gamma$ is torsion-free, then $\Phi$ is the holonomy group of the corresponding manifold. In the general case, $\Phi$ is still canonically determined by $\Gamma$
and we call it the holonomy group of $\Gamma$. Finally, if the group $\Gamma$ has no $p$-torsion, then the preimage of any $p$-subgroup $S$ of $\Phi$
is torsion-free, and therefore is a Bieberbach group providing a flat manifold with holonomy $S$.

Additional geometric properties of the manifold can be checked from the
associated representation $\rho : \Phi \to \GL_N(\C)$ of the group $\Phi$.
For instance it is K\"ahler if and only if $N$ is an even integer and each of the irreducible
constituents of $\rho$ which have orthogonal (or real) representation
type appear an even number of times (see \cite{JOHNSONREES}).

\subsection{Complex braid groups}

We refer the reader to \cite{BMR} for the construction of complex braid groups.
Let $W < \GL_n(\C)$ be a complex reflection group. We let $\mathcal{A}$ denote the
set of hyperplanes fixed by the reflections of $W$ (so-called reflecting
hyperplanes), and denote $X = \C^n \setminus \bigcup \mathcal{A}$ their
complement. We denote $\mathcal{R}$ the set of (pseudo-)reflections in $W$. We set $\mathcal{R}^*$ its subset of distinguished reflections, that is (pseudo-)reflections
$s$ with eigenvalues $\{ 1, e^{2 \pi \ii/m} \}$ where $m$ is the order of the cyclic
subgroup of $W$ fixing $\Ker(s-1) \in \mathcal{A}$. There is a natural correspondance
$\mathcal{R}^* \leftrightarrow \mathcal{A}$ given by $s \mapsto \Ker(s-1)$. It is $W$-equivariant w.r.t. the conjugation action on $\mathcal{R}^*$ and the permutation action
on $\mathcal{A}$. 

By definition, neglecting base points, we have $P = \pi_1(X)$ and $B = \pi_1(X/W)$.
The short exact sequence $1 \to P^{ab} \to B/(P,P) \to W \to 1$
is exact, and $P^{ab} = \pi_1(X)^{ab} = H_1(X,\Z)$ is a free $\Z$-module of finite rank, $H_1(X,\Z) \simeq \Z^N$
where $N = |\mathcal{A}|$ is the number of reflecting hyperplanes.
It admits a basis $(c_H)_{H \in \mathcal{A}}$ uniquely defined
by the condition $\int_{c_{H_1}} \omega_{H_2} = \delta_{H_1,H_2}$,
where $\omega_H = \frac{1}{2  \pi\ii} \dd \varphi_H/\varphi_H$ is the logarithmic
1-form associated to the hyperplane $H = \Ker \varphi_H$,
and $\delta_{H_1,H_2}$ is the Kronecker symbol.
Moreover, the conjugation action of $B/(P,P)$ on $P^{ab}$
factorizes through $W$ and coincides with the permutation action
of $W$ on $\mathcal{A}$ under the correspondance $H \leftrightarrow c_H$. The kernel of this permutation action
is exactly the set of all elements of $W$ which commute
with all the distinguished pseudo-reflections. Since these
reflections generate $W$, this kernel is thus exactly the centre of $W$.

Let $\overline{W} = W/Z(W)$, and $P_0$ the subgroup of $B/(P,P)$ generated
by $P^{ab}$ and by the (image of the) element $\mathbf{z} \in B$ defined as the homotopy class
of $t \mapsto e^{2  \pi\ii t/d} x_0$, where $x_0$ is the chosen base point in $X$ and $d = |Z(W)|$.
This element is central and its image 
in $W$ generates $Z(W)$ when
$W$ is irreducible
(see \cite {BMR}). Moreover $\mathbf{z}^d = \boldsymbol{\pi} : t \mapsto e^{2 \pi \ii t} x_0$ is a central element in $P$,
that can be written $\sum_{H \in \mathcal{A}} c_H$ inside $P^{ab}$.
If $W = W_1 \times \dots \times W_r$ is a decomposition of $W$ in irreducible components,
and $B = B_1 \times \dots \times B_r$ is the corresponding decomposition of $B$,
let $\mathbf{z}_1,\dots,\mathbf{z}_r$ denote the associated central elements. We let $Z_0(B)$
denote the central subgroup of $B$ generated by $\mathbf{z}_1,\dots,\mathbf{z}_r$.
Although we will not use it here it can be shown that
$Z_0(B) = Z(B)$ and that
the projection $B \to W$
induces a short exact
sequence $1 \to Z(P) \to Z(B) \to Z(W)\to 1$
(see \cite{DMM}). However, we will keep the notation $Z_0(B)$ in order to emphasize that our proofs do
not make use of the classification of complex reflection groups, as opposed to the proof that $Z(B) = Z_0(B)$.

\begin{theorem}
For every 
complex reflection group $W$, the group $B/(P,P)$ is crystallographic with holonomy
group $W/Z(W)$ of dimension $N = |\mathcal{A}|$.
The kernel of the projection map $B/(P,P) \to W/Z(W)$ is the subgroup $P_0$ generated by $P^{ab}$ 
and $Z_0(B)$.
We have $P_0 \simeq \Z^N$.
\end{theorem}
\begin{proof}

Decomposing $W$ into irreducibles $W_1\times \dots \times W_r$,
we get, with obvious notations, that 
$P \simeq \prod_{1 \leq k \leq r} P_k$,
$B/(P,P) \simeq \prod_{1 \leq k \leq r} B_k/(P_k,P_k)$,
$Z_0(B) \simeq \prod_{1 \leq k \leq r} Z_0(B_k)$ and from
this one easily checks that we can
assume w.l.o.g. that $W$ is irreducible.

Clearly the kernel of the composite map $B/(P,P) \to W \to \overline{W}$
contains $P^{ab}$ and $\mathbf{z}$, and therefore $P_0$. Conversely,
if $b \in B/(P,P)$ has trivial image inside $\overline{W}$, then its image $\overline{b} \in W$ belongs to $Z(W)$.
There exists $c \in \langle \mathbf{z} \rangle \subset P_0$ whose image $\overline{c}$ inside $W$ is equal to
$\overline{b} $
hence $c^{-1} b \in P^{ab}$ and $b \in P_0$. Therefore we have
a short exact sequence $1 \to P_0 \to B/(P,P) \to \overline{W} \to 1$. Since $P_0$ is abelian
we have an action $\overline{W} \to \Aut(P_0)$. Since $w \in W$ acts trivially on $P^{ab} \subset P_0$
if and only if $w \in Z(W)$, this action is faithful. It remains to prove $P_0 \simeq \Z^N$.
 
We identify $P^{ab}$ with $\Z^N$ by using the basis $(c_H)_{\mathcal{A}}$ and an arbitrary
total ordering on $\mathcal{A}$,
and we let $d = |Z(W)|$. Then $\mathbf{z}^d = \boldsymbol{\pi} \in P^{ab}$ corresponds to
the vector $(1,\dots,1) \in \Z^N \simeq P^{ab}$.
By definition $P_0$ is a quotient of $P^{ab} \times \langle \mathbf{z} \rangle \simeq \Z^N\times \Z \simeq \Z^{N+1}$
 and the kernel of $\Z^{N+1} \simeq P^{ab} \times \langle \mathbf{z} \rangle \onto P_0$
can be identified with the line spanned
by the vector $ v = (1,1,\dots,1,-d)$, because $\mathbf{z}^m \in P^{ab}$ iff $d$ divides $m$.
Therefore, $P_0 \simeq \Z^{N+1}/\Z v \simeq \Z^N$ and this concludes the proof.
 
\end{proof}

In particular we get the following characterization of when $B/(P,P)$
is crystallographic with holonomy group $W$. This of course includes the case
of the ordinary braid group.

\begin{corollary} The group $B/(P,P)$ is crystallographic with holonomy group $W$ 
if and only if
$Z(W) = 1$.
\end{corollary}
Of course we only need to analyse this condition $Z(W) = 1$ for \emph{irreducible}
groups. We refer the reader to \cite{LEHRTAY} for the Shephard-Todd classification of
such groups (see also table \ref{tab:exceptions}).
Among irreducible \emph{Coxeter} groups and since $Z(W) \subset \R \cap \mu_{\infty}(\C) = \{-1,1 \}$,
this condition is equivalent to the condition $-1 \not\in W$. It is well-known that
this happens exactly in Coxeter types $A_n,n\geq 2$, $D_{2n+1}, n \geq 1$, $I_2(2m+1), m \geq 1$ and $E_6$.
Among non-real irreducible complex reflection groups, we have $Z(W) \neq 1$
for all exceptional ones. The irreducible groups inside the family $G(de,e,n)$ have center
of order $d(e \wedge n)$. Therefore, the only possibility for $Z(W) = 1$ is $d=1$ and $e \wedge n = 1$.
Thus, the non-real irreducible complex reflection groups with $Z(W) = 1$ are
the $G(e,e,n)$ with $e \wedge n =1$. Note that this includes the real groups $A_n, n \geq 2$,
$D_{2n+1}, n \geq 1$ and $I_2(2m+1), m \geq 1$. 
\begin{corollary} If the irreducible components of $W$ are of type $G(e,e,n)$ with $e \wedge n =1$,
or of type $E_6$, then $B/(P,P)$ is crystallographic with holonomy group $W$.
\end{corollary}

\medskip

We now prove that this construction is compatible with parabolic inclusions.

\begin{proposition} \label{prop:injparab} Let $W'$ be a parabolic subgroup of $W$, and $B'$ (resp. $P'$) the corresponding (pure)
braid group. There is an embedding $B'/(P',P') \into B/(P,P)$, canonical up to
$P^{ab}$-conjugacy.
\end{proposition}
\begin{proof}
We consider an embedding $\iota : B' \into B$ as defined in \cite{BMR}. Such embeddings
are canonical up to $P$-conjugacy. By composition with the canonical projection $B \to B/(P,P)$
we get a morphism $\bar{\iota} : B' \to B/(P,P)$. We want to show $\Ker\, \bar{\iota} = (P',P')$.
Since $\iota(P') \subset P$ we get $\Ker\, \bar{\iota} \supset (P',P')$.

Let $x \in \Ker \, \bar{\iota}$. We have $\iota(x) \in (P,P)$ hence
$\iota(x) \in (P,P) \cap \iota(B')$ and we need to show $(P,P) \cap \iota(B') = (P',P')$.
By commutation of the diagram
$$
\xymatrix{
1 \ar[r] & P' \ar[d] \ar[r] & B' \ar[d]_{\iota} \ar[r] & W' \ar[d]\ar[r] & 1 \\
1 \ar[r] & P \ar[r] & B \ar[r] & W \ar[r] & 1 \\
}
$$
we have $\iota(B') \cap P = \iota(P')$, and we thus need to show
$(P,P) \cap \iota(P') = (P',P')$. For this we need to recall that $j = \iota_{|P'}$
is constructed from
\begin{enumerate}
\item the choice of an open ball $\Omega$ in $\C^n$
\item a base point $x_2 \in \Omega \cap \tilde{X}'$
\end{enumerate}
where $\tilde{X}' \simeq X' \times \C^{m}$
is the complement in $\C^n$ of the reflecting hyperplanes
of $W' \subset W$. From this there is an isomorphism $\pi_1(X \cap \Omega,x_2) \to \pi_1(\tilde{X}',x_2)$.
Letting $\bar{x}_2$ denote the natural projection of $x_2$ on $X'$, we have an isomorphism $\pi_1(\tilde{X'}, x_2) =
\pi_1(X'\times \C^m, x_2) \to \pi_1(X',\bar{x}_2)$. These isomorphisms identify $P'=\pi_1(X',\tilde{x}_2)$ and $\pi_1(X \cap \Omega,x_2)$.
The embedding of $P'$ into $P = \pi_1(X,x_2)$ is then induced by the inclusion $X \cap \Omega \subset X$.
Now, the commutator subgroup $(P,P)$ is the kernel of the Hurewicz morphism $\pi_1(X,x_2) \to H_1(X,\Z)$.
Functoriality of this morphism applied to the inclusion $(X \cap \Omega, x_2) \subset (X,x_2)$ yields a commutative diagram
$$
\xymatrix{
1 \ar[r] & (P',P') \ar[d] \ar[r] & P' \ar[d]_{\iota} \ar[r] & H_1(X',\Z) \ar[d]\ar[r] & 1 \\
1 \ar[r] & (P,P) \ar[r] & P \ar[r] & H_1(X,\Z) \ar[r] & 1. \\
}
$$
Finally, it is known that $H_1(X',\Z) \simeq H_1(X' \times \C^m,\Z) = H_1(\tilde{X}',\Z)\to H_1(X,\Z)$
is an embedding, by the basic homological theory of hyperplane arrangements (see e.g. \cite{ORLIKTERAO}).
This injectivity implies $(P,P) \cap \iota(P') = (P',P')$, and this concludes the proof of the proposition.

\end{proof}

\subsection{Bieberbach subgroups from 2-subgroups}

If $G$ is a subgroup of $W$, we let $\mathfrak{P}_W(G)$ denote its inverse image
under the map $B/(P,P) \to W$. If $\bar{G}$ denote the image of $G$ inside
$\overline{W} = W/Z(W)$, we have a short exact sequence $1 \to P_0 \to \mathfrak{P}_W(G) \to \bar{G} \to 1$, and the group $\mathfrak{P}_W(G)$ is again crystallographic, of dimension $|\mathcal{A}|$
and holonomy group $\bar{G}$.

\begin{theorem} \label{thm:notwo} 
For every complex reflection group $W$,
the group $B/(P,P)$ has no element of order $2$.
\end{theorem}
\begin{proof}
Decomposing $W$ into irreducibles $W_1\times \dots \times W_r$,
we get $B/(P,P) \simeq \prod_{1 \leq k \leq r} B_k/(P_k,P_k)$ and thus we
can assume w.l.o.g. that $W$ is irreducible.

Assume by contradiction that there exists $\beta \in B$ such that $\beta^2 \in (P,P)$ and $\beta \not\in (P,P)$.
Since $P^{ab}$ is torsion-free we have $\beta \not\in P$.
Let
$g$ denote the image of $\beta$ in $W$. Then $g^2 =1$ and $g \neq 1$. Therefore
$\C^n = \Ker(g-1) \oplus \Ker(g+1)$ with $E = \Ker(g-1)$, and
$E^{\perp} = \Ker(g+1) \neq 0$. By Steinberg's theorem, the parabolic
subgroup $W_0 = \{ w \in W ; w_{|E} = 1 \}$ is generated
by $\mathcal{R} \cap W_0$. Since $g \in W_0$, we have $W_0 \neq \{ 1 \}$
and therefore there exists $s \in \mathcal{R}^* \cap W_0$. We denote by
$H_0 = \Ker(s-1)$ its reflecting hyperplane.
We associate to each $H \in \mathcal{A}$ a linear form $\varphi_H \in (\C^N)^*$
with kernel $H$.
We have $E \subset H_0$. We denote by $x_0 \in X = \C^N \setminus \bigcup \mathcal{A}$
the chosen basepoint, that we write as $x_0 = x_E + x_E^{\perp}$ according to
the decomposition $\C^N = E \oplus E^{\perp}$. We let $a(t) = x_E + \exp( \pi\ii t) x_E^{\perp}$. Changing the base point amounts to
conjugating all the elements of $B$ we are interested in by an element of $P$,
and therefore this 
affects neither
our conditions on $\beta$ -- namely $\beta^2 \in (P,P)$ and $\beta \not\in (P,P)$ -- nor $g$.
 Therefore, up to replacing $x_E$
by some multiple of it, we can assume that $\varphi_H(x_E) \neq 0 \Rightarrow |\varphi_H(x_E)| 
>|\varphi_H(x_E^{\perp})|$. This implies $\varphi_H(a(t)) \neq 0$ for all $t$ and all $H$, and
therefore $a$ defines an element $\alpha$ of $B$. Its square $\alpha^2 \in P$ is (the class of) the loop
$t \mapsto x_E + \exp(2  \pi\ii t) x_E^{\perp}$. Finally, the class $\bar{\alpha}$ of $\alpha$ in $W$
is equal to $g$. Therefore, we can write $\beta  = x \alpha$ for some $x \in P$.
Then $1=\beta^2 = x\alpha x \alpha  = x \alpha x \alpha^{-1} \alpha^2
= x (g. x) \alpha^2$. We have $\alpha^2 = \sum_{H \in \mathcal{A}} a_Hc_H$
with $a_H = (1/2  \pi\ii)\int_{\alpha^2} \frac{\dd \varphi_H}{\varphi_H}$.
If $E \subset H$, we have $\varphi_H(\alpha^2(t)) = \exp(2  \pi \ii t) \varphi_H(x_E^{\perp}) \neq 0$
and
$$
a_H = \frac{1}{2  \pi\ii}\int_{\alpha^2} \frac{\dd \varphi_H}{\varphi_H} = \frac{1}{2  \pi\ii} \int_0^1 2  \pi\ii \dd t = 1.
$$
In particular, we get $a_{H_0} = 1$.
Writing $x = \sum_H u_H c_H$ the equation $1 = x(g.x) \alpha^2$ yields
$$
0 = \sum_H u_H c_H + \sum_H u_H c_{g(H)} + \sum_{H \in \mathcal{A}} a_H c_H.
$$ 
Since $g(H_0) = H_0$, and $g(H) = H_0 \Rightarrow H = g^2(H) = g(H_0) = H_0$
the coefficient of $c_{H_0}$ is $0=2 u_{H_0} + 1$ with $u_{H_0} \in \Z$, a contradiction.
\end{proof}

\begin{remark}
In case $W$ is a real reflection group, hence $B$ is an Artin group of finite Coxeter
type, a partly combinatorial variation of this proof can be given, using Richardson's classification
of involutions in Coxeter groups. Richardson's theorem (see \cite{RICHARDSON}) indeed states that, if $W$ has $I$
for set of Coxeter generators, then $g$ is up to conjugation equal to $w_J$, for $w_J$ the
longest element of the standard parabolic subgroup $W_J$ generated by $J \subset I$. Moreover,
this element has to be central inside $W_J$. Choosing for $\alpha$ an Artin generator of $B$
corresponding to some element of $J$, and writing $\beta = x\alpha$ we get the same contradiction.
\end{remark}

\medskip

As a corollary of the above theorem, and following the idea of \cite{GGO}, one gets
Bieberbach groups associated to each complex reflection group.

\begin{corollary}
If $G$ is a 2-subgroup of $W$, then $\mathfrak{P}_W(G)$ is a Bieberbach group,
of dimension $|\mathcal{A}|$ and holonomy group $\bar{G}$.
\end{corollary}
\begin{proof}
We already noticed that $\mathfrak{P}_W(G)$ is crystallographic. If $x \in \mathfrak{P}_W(G)$
had finite order, then $x^{|G|} \in P^{ab}$ would have finite order, whence $x^{|G|} =1$
since $P^{ab}$ is torsion-free. Therefore the order of $x$ is a power of $2$, hence $\mathfrak{P}_W(G) \subset B/(P,P)$ would contain an element of order $2$, contradiction.
\end{proof}

Note that this corollary is (almost) always non-void, because the order of an irreducible reflection reflection
group of rank at least $2$ is always even. Indeed, the classification easily implies that they all contain a reflection
of order $2$, except 
for a few exceptional groups,
which all happen to have finite order (see table \ref{tab:exceptions} below). 

\section{Possible torsion inside $B/(P,P)$} 
\label{sect:posstorsion}

\subsection{Finite order elements and freely regular numbers}

We recall that a vector $x \in X$ is called \emph{regular} (in the sense of Springer) if there exists $w \in W$ for which
$x$ is an eigenvector. The order $d$ of the corresponding eigenvalue is called a regular integer (w.r.t. $W$).
Such a $w$ is called a regular element. It has necessarily
order $d$. For basic properties of regular elements we refer the reader to \cite{SPRINGER} or \cite{LEHRTAY}.

We say that such $x,w,d$ are \emph{freely} regular if in addition $\langle w \rangle$ acts freely on $\mathcal{A}$.
In that case, since the action of $\langle w \rangle$ has in particular to be faithful, $d$ coincides with the order
of $w$ inside $W/Z(W)$.

\begin{proposition} \label{prop:freelyimpliesfinite} Let $d$ be a freely regular element with respect to $W$. Then there exists $b \in B/(P,P)$
of order $d$.
\end{proposition}
\begin{proof}
Let $w \in W$, $x \in X$ and $\zeta \in \C^{\times}$ of order $d$ such that
$w.x = \zeta x$ and there is no $H \in \mathcal{A}$ such that $w(H) = H$.
Up to raising $w$ to some power coprime to $d$, we can assume $\zeta = \exp(2  \pi\ii/d)$.
Let $\gamma(t) = \exp(2  \pi\ii t/d) x$. Since $x \in X$ and $X$ is defined by
linear (in)equations it defines a path $\gamma : [0,1] \to X$ joining $x$ and $\exp(2  \pi\ii/d) x = w.x$.
Therefore it defines a  class $[\gamma] \in B = \pi_1(X/W,\bar{x})$. It is straightforward to
check that $[\gamma]^d = \boldsymbol{\pi}$, and we recall that the image $w$ of $[\gamma]$ in $B/P = W$
has order $d$. Let us consider a collection $(u_H)_{H \in \mathcal{A}}$ of integers,
$p = \sum_H u_H c_H \in P/(P,P)$. Let $\tilde{p} \in B$ designate a lift of $p$, and set $\tilde{q} = [\gamma] \tilde{p} \in B$.
We denote $q$ the image of $\tilde{q}$ inside $B/(P,P)$. The image of $\tilde{q}$ in $B/P = W$ is $w$ and therefore has order $d$. It follows that the
order of $q$ is at least $d$. Since
$$
\tilde{q}^d = ([\gamma] \tilde{p})^d = \,^{[\gamma]} \tilde{p}.\,^{[\gamma]^2} \tilde{p}.\,^{[\gamma]^3} \tilde{p}.\dots . \,^{[\gamma]^d} \tilde{p}.[\gamma]^d
$$
we get
$
q^d = (w.p)(w^2.p)\dots (w^d.p) \bar{\boldsymbol{\pi}} =\sum_{H \in \mathcal{A}} c_H +   \sum_{k=1}^d \sum_{H \in \mathcal{A}}  u_H c_{w^k(H)} 
$. Let $G = \langle w \rangle \subset W$. We have $G \simeq \Z/d\Z$ and $G$ acts freely on $\mathcal{A}$. Let $\mathcal{B} \subset \mathcal{A}$
be a system of representatives of $\mathcal{A}/G$. We associate to each $H$ its orbit $\mathcal{O}(H) \in \mathcal{A}/W$. Then, the
above equation can be reformulated as
$$
q^d = \sum_{H \in \mathcal{A}} c_H + \sum_{H \in \mathcal{A}} \sum_{g \in G} u_{g(H)} c_H
=  \sum_{H \in \mathcal{B}} \sum_{J \in \mathcal{O}(H)}  c_J + \sum_{H \in \mathcal{B}} \sum_{J \in \mathcal{O}(H)} \sum_{g \in G} u_{g(J)} c_J.
$${}
Since the action is free, this is equal to
$$
 \sum_{H \in \mathcal{B}} \sum_{J \in \mathcal{O}(H)}  c_J + \sum_{H \in \mathcal{B}} \sum_{J \in \mathcal{O}(H)}\left( \sum_{K \in \mathcal{O}(H)} u_{K} \right) c_J
= \sum_{H \in \mathcal{B}}\left(\left( 1+\sum_{K \in \mathcal{O}(H)} u_{K} \right) \left( \sum_{J \in \mathcal{O}(H)}  c_J  \right) \right)
$$
and thus $q^d = 1$ if and only if, for every orbit $\mathcal{O} \in \mathcal{A}/G$, we have $\sum_{K \in \mathcal{O}} u_K = -1$.
Since this condition is easy to fulfill (e.g. take $u_H = -1$ for $H \in \mathcal{B}$ and $u_H = 0$ for $H \not\in \mathcal{B})$
there exists $q \in B/(P,P)$ of order $d$.

\end{proof}

In particular, the symmetric group $\mathfrak{S}_n$ contains a regular element $(1 \ 2 \ \dots \ n)$ of order $n$, which is freely
regular if and only if $n$ is odd. From this we recover proposition 19 of \cite{GGO}, namely that $B_n/(P_n,P_n)$ has
elements of order $n$ if $n$ is odd, and therefore of order $k$ for $k \leq n$ odd.
Note also that, if $W$ is decomposed into a sum of irreducibles $W_1 \times \dots
\times W_r$, then every collection $(b_1,\dots,b_r)$ of elements of
$B_k/(P_k,P_k)$, $k=1,\dots,r$ of finite orders $m_1,\dots,m_r$,
defines an element of order $m = \lcm(m_1,\dots,m_r)$ of the group
$B_1/(P_1,P_1)\times \dots \times B_r/(P_r,P_r) \simeq B/(P,P)$.
It follows that, if $W$ admits a parabolic subgroup $W'$ which is a sum
$W_1 \times \dots \times W_r$ of such irreducibles, then by proposition
\ref{prop:injparab} we know $B/(P,P)$
also contains elements of order $m = \lcm(m_1,\dots,m_r)$,
and even a subgroup isomorphic to $(\Z/m_1 \Z) \times \dots \times (\Z/m_r\Z)$.

In particular, the results above reprove theorem 3 and theorem 6 of \cite{GGO}.
This is clear for theorem 3. For theorem 6, since a transitive
abelian permutation group has the same order as its degree, we know that
every abelian subgroup $(\Z/m_1 \Z) \times \dots \times (\Z/m_r\Z)$ possibly embedding in
$\mathfrak{S}_n$ must satisfy $m_1+\dots+m_r \leq n$ (see e.g. \cite{HOFFMAN}, prop. 2), and in particular,
since we know how to build such a group from a parabolic $\mathfrak{S}_{m_1}\times\dots\times \mathfrak{S}_{m_r}$ this reproves that the isomorphism
types of abelian subgroups that can be embedded into $B_n/(P_n,P_n)$
are exactly the same that can be embedded into 
$\mathfrak{S}_n$.

Propositions \ref{prop:freelyimpliesfinite} and \ref{prop:injparab}, combined together, produce lots of elements of finite order inside $B/(P,P)$. Note that proposition \ref{prop:freelyimpliesfinite} alone
is not enough to produce all the possible orders : there are parabolic subgroups $W_0 \subset W$
where $W_0$ admits a freely regular degree $d$ which is not a freely regular degree for $W$
(for example $W_0$ of type $E_7 = G_{36}$ as a parabolic subgroup of $W$ of type $E_8 = G_{37}$, see
section \ref{sec:tableexc} and table \ref{tab:exceptions}), and by proposition \ref{prop:injparab} this provides
an element of degree $d$ inside $B$ not directly produced by proposition \ref{prop:freelyimpliesfinite}.

Nevertheless, proposition \ref{prop:freelyimpliesfinite} admits a partial converse.

\begin{proposition} \label{prop:finiteimpliesfreely} Let $d$ be a regular number for $W$. If $B/(P,P)$ has an element
of order $d$ whose image in $W$ is regular, then $d$ is freely regular.
\end{proposition}
\begin{proof} Let $x$ be such an element, and let $g_0 \in W$ denote its image. We can choose
the base-point so that it is a regular eigenvector for $g_0$, and therefore we get as before
a lift $g \in B/(P,P)$ with $g^d = \boldsymbol{\pi} = \sum_{H \in \mathcal{A}} c_H$. Now $x = g y$ for some
$y = \sum_H y_H c_H$ in $P^{ab}$, with $y_H \in \Z$. Then, $x^d = (g_0.y)(g_0^2.y)\dots(g_0^{d-1}.y) y g^d$ is equal to
$$
\sum_{H \in \mathcal{A}} \left( \frac{|G|}{|G_H|} \left( \sum_{J \in G.H} y_J \right) +1 \right) c_H
$$
where $G = \langle g_0 \rangle$, $G_H \subset G$ is the fixer of $H \in \mathcal{A}$,
and $G.H$ denote its orbit. Thus $x^d =1$ implies $|G| = |G_H|$ for all $H$, thereby
proving that $d$ is freely regular.
\end{proof}

\begin{corollary}
Let $G$ is a subgroup of $W$ such that
\begin{enumerate}
\item no freely regular prime number divides $|G|$
\item for every odd prime number $p$ dividing $|G|$, all elements of order $p$ are regular
\end{enumerate}
Then, $\mathfrak{P}_W(G)$ is a Bieberbach group.
\end{corollary}
\begin{proof}
Assume that $x$ is a finite-order element in $\mathfrak{P}_W(G)$. Up to raising it
to some power, we can assume this order $p$ is a prime number. By theorem \ref{thm:notwo} it
is odd. Let $\bar{x}$ denote its image in $G \subset W$. Since $P^{ab}$ is torsion-free,
the order of $\bar{x}$ is equal to $p$. Since it divides $|G|$, it is a regular number by (2),
and also $\bar{x}$ is a regular element.
By proposition \ref{prop:finiteimpliesfreely} this implies that $p$ is freely regular, contradicting (1). This proves the claim.
\end{proof}

Of course these conditions are quite strong. They are nevertheless satisfied in some cases,
which enables us at least to construct some flat manifolds whose holonomy groups
exceed the class of  2-groups.

\begin{corollary} \label{cor:G4} If $W = G_4$, then $B/(P,P)$ is a Bieberbach group, of dimension $4$.  If $W = G_{6}$, then $B/(P,P)$ is a Bieberbach group, of dimension $10$. Both have for holonomy group the alternating group $\mathfrak{A}_4$.
\end{corollary}
\begin{proof} We first assume $W = G_4$. One has $|W| = 24 = 2^3.3$, the regular numbers are $1, 2, 3, 4, 6$
(see section \ref{sec:tableexc} and table \ref{tab:exceptions}), and one checks easily that there are no freely regular numbers. 
The pseudo-reflections of $G_4$ have order $3$, and by the Sylow theorems
every element of order $3$ is conjugated to one of them. Since $3$ is a regular number
they are regular and condition (2) of the previous corollary is satisfied. Since $|\mathcal{A}| = 4$
and $\overline{W} = \mathfrak{A}_4$ the conclusion follows. Now assume $W = G_6$.
Then $|W| = 2^4.3$, $3$ is regular, again there are no freely regular numbers,
and the same argument applies since one class of reflections for $W$ has order $3$.
Since $|\mathcal{A}| = 10$
and $\overline{W} = \mathfrak{A}_4$ the conclusion follows.
\end{proof}

It can be checked that there are no other exceptional group on which the above
criterion can be applied, with $G = W$.

\subsection{Determination of the freely regular numbers}
In case $W$ is a real reflection group, the freely regular numbers are fairly easy to determine, because of the next proposition.

\begin{proposition} \label{prop:freeregreal} If $W$ is a real reflection group, then its freely regular numbers are exactly its \emph{odd} regular numbers.
\end{proposition}
\begin{proof}
Let $g \in W$ be regular of order $d$. If $d$ is freely regular, it has to be odd by theorem \ref{thm:notwo} and proposition
\ref{prop:freelyimpliesfinite}.
Assume
conversely that $d$ is odd, and choose a root system $R = \{ \pm v_H, H \in \mathcal{A} \}$. Assume by
contradiction that the action of $g$ on $\mathcal{A}$ is not free. This means that $g. v_H = \pm v_H$
for some $H$, and thus $g^2.v_H = v_H$. Since $g^2$ is still regular of order $d$, this contradicts
the freeness of its action on $R$, known by \cite{SPRINGER} proposition 4.10 (i) -- note that, although this freeness statement is stated there only for rational reflection groups, the proof is valid for an arbitrary real reflection group.
\end{proof}

In case of complex reflection groups, the situation is more complicated. Let us introduce the order $\kappa(W)$
of the extension $1 \to P/(P,P) \to B/(P,P) \to W \to 1$ as an element of $H^2(W,P^{ab})$. This integer,
which originally appeared in \cite{ARREFL} as the periodicity of a monodromy representation of $B$,
and was subsequently identified by V. Beck in \cite{BECK} as the order of this extension, can be described as follows.
First introduce, for every reflecting hyperplane $H \in \mathcal{A}$, the parabolic subgroups $W_H = \{ w \in W \ | \ w_{|H} = \Id_H \}$ and
$C(H) = \{ w \in W \ | \ w_{|H^{\perp}} = \Id_{H^{\perp}} \}$.  They are normal subgroups of $N(H) = \{ w \in W \ | \ w(H) = H \}$
and we have $W_H \cap C(H) = \{ \Id \}$.
We set $e_H = |W_H|$, $f_H = |N(H)/C(H)|$. Clearly $e_H$ divides $f_H$. Then $\kappa(W) = \lcm \{ f_H ; H \in \mathcal{A} \}$ (\cite{BECK}, corollary 1). It is a general fact that $|Z(W)|$ divides $\kappa(W)$ (see \cite{ARREFL}, cor. 5.10).

In case $W$ is a Coxeter group, we have $\kappa(W) = 2$ (see  \cite{ARREFL}, theorem 6.4 or \cite{BECK} remark 2).
If $W = G(de,e,n)$ with $n \geq 2$, then (see \cite{ARREFL} proposition 6.1, corrected in \cite{BECK})
\begin{itemize}
\item $\kappa(W)=2de$ if $de$ is odd and $n \geq 3$.
\item $\kappa(W) = de$ if $d\neq 1$, $n=2$, $de$ even.
\item $\kappa(W) = de$ if $n\geq 3$, $de$ even.
\item $\kappa(W) = 2de$ if $d \neq 1$, $n=2$ and $de$ is odd.
\item $\kappa(W) = 2$ if $d=1$ and $n=2$.
\end{itemize}
Note that $\kappa(W)$ is an even integer for every complex reflection group.
Moreover, when $W = G(de,e,n)$ and $W$ is not of the form
$G(e,e,2)$, we have $\kappa(W) = \lcm(2, de)$. In particular this formula
is always valid in rank at least 3, and also when $W$ is \emph{not} a real reflection group.

\begin{theorem} Let $d$ be a regular number. 
The integer $d$ is coprime to $\kappa(W)$ if and only if $d$ is freely regular.
\end{theorem}
\begin{proof}
Let $w \in W$ be regular of order $d$, and $v \in X$ such that $w.v = \zeta v$,
$\zeta$ a primitive $d$-th root of $1$. If we assume by contradiction that
$w$ is not freely regular, then
there exists $H \in \mathcal{A}$, $k \in \{1,\dots, d-1 \}$ such that
$w^k(H) = H$. Then $w^{k f_H} \in C(H)$. 
We now borrow an argument which is apparently due to Kostant (see \cite{SPRINGER}, proof of proposition 4.10) for the case of Coxeter groups.
Let $v_2$ be a spanning vector
for $H^{\perp}$, and $(\ | \ )$ a $W$-invariant hermitian scalar product on $V$.
Then $(v|v_2) = (w^{k f_H} v | w^{k f_H} v_2) = (\zeta^{k f_H} v |v_2)$.
Since $v \in X$ we have $v \not\in H$, that is $(v|v_2) \neq 0$. This
implies $\zeta^{k f_H} = 1$.
 But since $d$ is coprime to $\kappa(W)$,
it is in particular coprime to $f_H$, hence $\zeta^{k f_H} = 1 \Rightarrow \zeta^k = 1$. This means $w^k. v = v$. But since $w^k \neq 1$ and $v \in X$, this
contradicts Steinberg's theorem and concludes the direct part of the proof.

We have no direct proof of the converse implication, but we can check it
by using the classification. We check it case by case on the exceptional
groups (by computer means), see the table below.
For the general series, we use an explicit description of the regular elements,
that we postpone to
section \ref{sect:freeregsgdeen} below. The conclusion of the theorem is then
consequence of the case of real reflection groups (proposition \ref{prop:freeregreal})
together with lemmas \ref{lem:freeregGdeen} and \ref{lem:freeregGeen}
proved there.
\end{proof}

\subsection{Freely regular elements in types $G(de,e,n)$}
\label{sect:freeregsgdeen}

We first consider the case $W =G(d,1,n)$, $n \geq 2$ and $d > 1$.
Let $x_0=(1,\zeta_{dn}^{-1},\zeta_{dn}^{-2},\dots,\zeta_{dn}^{-(n-1)})$,
were $\zeta_k = \exp(2  \pi\ii/k)$. Since $n \geq 2$ we have $x_0 \in X$.
Let now $g_0 \in \GL(V)$ be defined by $g_0.e_i = e_{i+1}$
if $i < n$, $g_0.e_n = \zeta_d e_1$. It is clear that $g_0 \in W$,
and that $g_0.x_0 = \zeta_{dn} x_0$. Therefore $g_0$ is a regular
element of $W$, associated to the primitive $dn$-th root of unity $\zeta_{dn}$.
Now, the degrees of $W$ as a reflection group are $d,2d,3d,\dots,nd$,
while its codegrees are $0,d,\dots,(n-1)d$ (see e.g. \cite{LEHRTAY} table D.5). By a well-known
(and useful) criterion (see \cite{LEHRTAY}, 11.28), the regular numbers
being the numbers which divide an equal number of degrees and codegrees,
it follows
that the regular numbers for $W$ are the divisors of $nd$. Therefore,
the regular elements of $W$ are the conjugates of the powers of $g_0$.
We write $g_0 = \delta \sigma$ where $\sigma  = (1\ 2 \ \dots \ n) \in \mathfrak{S}_n$ and $\delta = \diag(\zeta_d,1,\dots,1)$. We have $\sigma = p_{\Sigma}(g_0)$
where $p_{\Sigma} : G(d,1,n) \to \mathfrak{S}_n$ is the natural projection.
Now note that the orthogonal of a reflecting hyperplane is spanned by
a vector having at most 2 non-zero entries. Therefore, $g_0^k$ may stabilize
a reflecting hyperplane only if $\sigma^k$ has a 1-cycle or a 2-cycle in its
decomposition into disjoint cycles. We write $k = nq + k_0$, $k_0 <n$, $q \geq 0$.
Since $\sigma^k = \sigma^{k_0}$, this is possible only if $k_0 = 0$ or $k_0 = n/2$.
Note that $g_0^n = \zeta_d \Id \in Z(W)$ stabilizes every hyperplane and
therefore so do the $g_0^k$ when $k_0=0$. Similarly, if $k_0 = n/2$,
then $g_0^k = \zeta_d^q g_0^{n/2}$ stabilizes the same hyperplanes as $g_0^{n/2}$. But since $\sigma^{n/2}$ is a disjoint product of transpositions, these hyperplanes should be of the form $z_{r+(n/2)} = \alpha z_r$
for some $\alpha,r$. It is straightforward to check that such an hyperplane is stable
if and only if $\alpha^2 = \zeta_d^{-1}$ and there exists such an $\alpha$ exactly
when $d$ is odd.

Let now $n \geq 2$, $e > 1$ and $d>1$. We set $W = G(de,e,n)$. The
reflecting hyperplanes of $W$ are the same as for $G(de,1,n)$,
and therefore the (freely acting) regular elements for $G(de,e,n)$
are exactly the (freely acting) regular elements for $G(de,1,n)$ which belong
to $G(de,e,n)$. Recall that $G(de,e,n)$ is the subgroup of the $g \in G(de,1,n)$
such that $\Pi(g) \in \mu_d$, where $\Pi : G(de,e,n) \to \mu_{de}$
is the morphism obtained by multiplying together all the non-zero
entries. So it is sufficient to check when $g_0^k \in G(de,e,n)$.
Since $\Pi(g_0^k) = \Pi(g_0)^k = (\zeta_{de})^k$,
we have $g_0^k \in G(de,e,n)$ iff $(\zeta_{de})^{kd} = 1$ iff $\zeta_e^k = 1$
iff $e$ divides $k$.

From this study we get the following characterization of freely regular numbers.

\begin{lemma}  \label{lem:freeregGdeen} Let $W = G(de,e,n)$, with $d >1$ and $n \geq 2$.
\begin{enumerate}
\item The regular numbers of $W$ are the divisors $m$ of $dn$.
\item The freely regular numbers of $W$ are the regular numbers coprime to $lcm(2,de)$.
\end{enumerate}
\end{lemma}
\begin{proof}
We already proved (1), as the condition $e$ divides $k$ above means that $e$ divides $den/m$,
that is $m$ divides $dn$. For (2) let $g \in W$ be regular of order $m$ with $m$ freely regular.
Up to conjugating $g$, we can assume that $g \in \langle g_0 \rangle \simeq \Z/(den\Z)$.
Since $\langle g_0^n \rangle$ has order $de$ we have $\langle g \rangle \cap
\langle g_0^n \rangle = \{ 1 \}$ if and only if $m$ is coprime to $de$. When $de$ is even,
we have $de = \lcm(2,de)$ and $\langle g \rangle$ acts freely if and only if
$\langle g \rangle \cap
\langle g_0^n \rangle = \{ 1 \}$, so this proves (2) in this case. When $de$ is odd,
we have $\lcm(2,de) = 2de$ and
we proved that $\langle g \rangle$ acts freely if and only if,
\begin{itemize}
\item $\langle g \rangle \cap \langle g_0^n \rangle = \{ 1 \}$ when $n$ is odd
\item $\langle g \rangle \cap \langle g_0^{\frac{n}{2}} \rangle = \{ 1 \}$ when $n$ is even.
\end{itemize}
The first condition is equivalent to saying that $m$ is coprime to $de$. Since $m$ divides
$den$ it implies that $m$ divides $n$ and therefore $m$ is odd, whence $m$ is coprime to $\lcm(2,de)=2de$.
The second condition is equivalent to saying that $m$ is coprime to the order $2de$ of
$\langle g_0^{\frac{n}{2}} \rangle$. This proves the claim.
\end{proof}

We now consider the case $W = G(e,e,n)$, $n \geq 3$, $e > 1$. The degrees of $W$
are $e,2e,\dots,(n-1)e, n$ and its codegrees are $0,e,\dots,(n-2)e, (n-1)e - n$.
By the preceedingly mentioned criterion, a regular number is then either
a divisor of $(n-1)e$ or a divisor of $n$. It is thus sufficient to describe
a regular element of order $(n-1)e$ and one of order $n$. We get the first one
by embedding $G(e,1,n-1)$ into $G(e,e,n)$ through $g \mapsto (g, \Pi(g)^{-1})$.
Indeed, letting $\tilde{x}_0 = (x_0,0)$ and $\tilde{g}_0 = g_0 \oplus \zeta_e^{-1} \Id_1$,
where $x_0,g_0$ are as defined above for $W = G(e,1,n-1)$, we get $\tilde{g}_0.
\tilde{x}_0 = \zeta_{(n-1)e} \tilde{x}_0$ and $\tilde{x}_0 \in X$. Therefore
$\tilde{g}_0$ is a regular element of order $(n-1)e$, associated to $\zeta_{(n-1)e}$.

If $n$ divides $e$, the divisors of $n$ are divisors of $(n-1)e$ and we are done.
We now assume otherwise.
We first describe a regular
element of order $n$.

Let us write $n = \delta n'$, $e = \delta e'$, with $\delta = \gcd(n,e)$ (hence
$\gcd(e',n')=1$). We set $x_2 = (1, \zeta_n^{-1},\zeta_n^{-2},\dots,
\zeta_n^{-(n'-1)}) \in \C^{n'}$, and $g_2 \in \GL_{n'}(\C)$ defined
by $g_2.e_i = e_{i+1}$ for $i < n'$, $g_2.e_{n'} = \zeta_{\delta} e_{n'}$,
where $e_1,\dots,e_{n'}$ denotes the canonical basis of $\C^{n'}$. We have
$g_2.x_2 = \zeta_n x_2$. We choose now $\delta$ positive reals
$0<\la_1 < \la_2 < \dots < \la_{\delta}$, and let
$x_1 = (\la_1 x_2) \oplus (\la_2 x_2) \oplus \dots \oplus (\la_{\delta} x_2) \in \C^n$.
We have $x_1 \not\in X$ if and only if $\zeta_n^k \in \mu_e$ for some $1 \leq k < n'$. But this means $1 = \zeta_n^{ke} = \zeta_{n'}^{ke'}$
that is $n'$ divides $k e'$. Since $n'$ and $e'$ are coprime this implies $n'$ divides
$k$, a contradiction. Therefore $x_1 \in X$. Setting $g_1 = g_2 \oplus g_2 \oplus
\dots \oplus g_2 \in \GL_n(\C)$ we have $g_1.x_1 = \zeta_n x_1$.
Since $g_1 \in G(e,1,n)$ and $\Pi(g_1) = (\Pi(g_2))^{n/n'} = (\Pi(g_2))^{\delta} = \zeta_{\delta}^{\delta} = 1$, we have $g_1 \in G(e,e,n)$.
Therefore $g_1$ is a regular element of $W$, attached to $\zeta_n$.

Now $\tilde{g}_0^{n-1} = g_0^{n-1}\oplus \zeta_e^{-(n-1)} \Id_1
= \zeta_e \Id_{n-1} \oplus \zeta_e^{-(n-1)} \Id_1$ stabilizes the hyperplane $z_1 = z_2$, and no hyperplane is stabilized by $\tilde{g}_0^k$ if $(n-1)$ does not
divise $k$, unless $n$ and $e$ are odd, in which case the condition is that $k$ is not
a multiple of $(n-1)/2$.
Likewise, $g_1^k$ stabilizes a reflecting hyperplane if and
only if 
\begin{itemize}
\item $k$ is a multiple of $n'$, or
\item $n'$ is even, $k$ is a multiple of $n'/2$ and $\delta$ is odd.
\end{itemize}
This yields a determination of the freely refular numbers in this case:

\begin{lemma} \label{lem:freeregGeen} Let $W = G(e,e,n)$, with $e >1$ and $n \geq 3$.
\begin{enumerate}
\item The regular numbers of $W$ are the divisors of $n$ or $(n-1)e$.
\item The freely regular numbers of $W$ are the regular numbers coprime to $lcm(2,e)$.
\end{enumerate}
\end{lemma}
\begin{proof}
We already proved (1).
For (2) let $g \in W$ be regular of order $m$ with $m$ freely regular.
Up to conjugating $g$, we can assume that $g \in \langle \tilde{g}_0 \rangle$ 
or $g \in \langle g_1 \rangle$. In the first case the proof that the order of $g$
is freely regular if and only if it is coprime to $\lcm(2,e)$ is the same
as in the previous lemma. We assume now $g \in \langle g_1 \rangle \simeq \Z/n\Z$
and write as before $n = \delta n'$, $e = \delta e'$, with $\delta = \gcd(n,e)$.
When $\delta$ is even,
$\langle g \rangle$ 
acts freely if and only if $\langle g \rangle \cap
\langle g_1^{n'} \rangle = \{ 1 \}$. This condition means that $m$ is coprime to $\delta = \gcd(n,e)$.
Since $m$ divides $n$ and $\delta$ is even this is equivalent to saying that $m$ is coprime to $2$
and $e$, which proves the claim.

When $\delta$ is odd and $n'$ is odd, again $\langle g \rangle$ 
acts freely if and only if $\langle g \rangle \cap
\langle g_1^{n'} \rangle = \{ 1 \}$ and this condition still means that $m$ is coprime to $\delta = \gcd(n,e)$.
Since $m$ divides $n$ this is equivalent to saying that $m$ is coprime to $e$. But in this case, since $m$ divides $n'$ it has to be odd, and therefore this is equivalent to saying that $m$ is coprime to $\lcm(2,e)$.

Finally, assume that $\delta$ is odd and $n'$ is even. Then $\langle g \rangle$ 
acts freely if and only if $\langle g \rangle \cap
\langle g_1^{n'/2} \rangle = \{ 1 \}$. This condition means that $m$ is coprime to $2\delta = 2\gcd(n,e)$.
This is equivalent to saying that $m$ is coprime to $2$ and to $\delta$, and therefore
to $2$ and $e$. This concludes the proof.

\end{proof}

Since the groups $G(1,1,n)$ are the symmetric groups
and the groups $G(e,e,2)$ are the dihedral groups, which are
both real reflection groups, for which we know the freely
regular numbers by proposition \ref{prop:freeregreal}, this concludes our study.

\subsection{Some non-abelian finite subgroups of $B/(P,P)$}

Let $G$ be a finite subgroup of $B/(P,P)$.
Since $P_0$ is torsion-free, the projection map $\pi : B/(P,P) \to \overline{W}$
identifies $G$ with a subgroup of $\overline{W}$. Conversely, if $Q$
is a subgroup of $\overline{W}$, it is the projection of
a finite subgroup of $B/(P,P)$ if and only if the induced short exact
sequence $1 \to P_0 \to \pi^{-1}(Q) \to Q \to 1$ is split. In cohomological
terms, it means that the image of the cohomology class $c \in H^2(\overline{W},P_0)$ inside $H^2(Q,P_0)$ is zero.

\begin{proposition} Assume that $Z(W)=1$, and that $G$ is a finite subgroup
of $W$ acting freely by conjugation on $\mathcal{R}^*$. Then there
is a finite subgroup of $B/(P,P)$ isomorphic to $G$ through $\pi$.
\end{proposition}
\begin{proof}
Since $Z(W) = 1$, we have $\overline{W} = W$, $P_0 = P/(P,P) = \Z \mathcal{A}$,
and the action of $W$ on $\Z \mathcal{A} = \Z \mathcal{R}^*$ is
defined by the conjugation action on $\mathcal{R}^*$. Since it is $G$-free,
we can decompose $\mathcal{R}^*$ into $G$-orbits $C_1 \sqcup \dots \sqcup C_m$ such that the $G$-action on each $C_k$ is free and transitive.
Then, considering the action by translation of $G$ on $\Z G$, we
have $H^2(G,\Z C_k) \simeq H^2(G,\Z G) \simeq H^2(G, \Ind_{\{1\}}^G \un)
\simeq H^2(\{ 1 \}, \Z) = 0$ and therefore $H^2(G,\Z \mathcal{A}) = 0$.
Therefore the sequence $1 \to P^{ab} \to \pi^{-1}(G) \to G \to 1$
splits and this concludes the proof.
\end{proof}

From this we recover theorem 7 of \cite{GGO}.

\begin{figure}
\begin{tikzpicture}[dba/.style={double,double equal sign distance}]
\draw (0,0) node (1) {$6$};
\draw (1)++(.5,.8660254037) node (2)  {$3$};
\draw (1)++(1,0) node (3)  {$1$};
\draw (2)++(.5,.8660254037) node (4)  {$5$};
\draw (2)++(1,0) node (5) {$2$};
\draw (3)++(1,0) node (6)  {$4$};
\draw (3)++(0,.5773502691) node (7) {$7$};
\draw (7) circle (12pt);
\draw (1)--(2);
\draw (1)--(3);
\draw (1)--(7);
\draw (2)--(7);
\draw (2)--(4);
\draw (3)--(7);
\draw (3)--(6);
\draw (4)--(7);
\draw (4)--(5);
\draw (5)--(6);
\draw (5)--(7);
\draw (6)--(7);
\end{tikzpicture}
\caption{Fano plane.}
\label{fig:fano}
\end{figure}
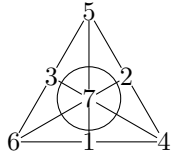

\begin{corollary} (\cite{GGO}, theorem 7) The group $B_7/(P_7,P_7)$ contains a finite non-abelian subgroup of order $21$.
\end{corollary}
\begin{proof} Let $G$ be a subgroup of order 21 of the collineation group of the Fano plane generated by a 3-fold symmetry $s$ fixing a point and a cyclic permutation $t$ of all 7 points, satisfying $st = t^2 s$, for instance $s = (1,2,3)(4,5,6)$ and
$t = (1,3,4,2,5,6,7)$ (see figure \ref{fig:fano}).
Then it is straightforward to check that no element of $G$ can commute with a transposition, so we can apply the proposition to get the result.
\end{proof}

More generally, we have the following result.

\begin{corollary} If $p$ is a prime number with $p \equiv 3 \mod 4$,
then $B_p/(P_p,P_p)$ contains a  Frobenius group of order $p(p-1)/2$,
non-abelian if $p > 3$.
\end{corollary}
\begin{proof}
Let $x$ denote a generator of $\F_p^{\times} \simeq \Z/(p-1)\Z$. We let
$G$ denote the group of affine transformations of $\F_p$
of the form $z \mapsto x^{2k}z+\beta$, $\beta \in \F_p$. Since $x^2$ has
order $(p-1)/2$, $G$ has order $p(p-1)/2$, and is non-abelian when $p>3$.
Now $G$ acts faithfully on $\F_p$ and can therefore be considered
as a subgroup of $\mathfrak{S}_p$. We now prove that is acts freely
on transpositions. Let us choose $g \in G \setminus \{ 1 \}$,
that is $g : z \mapsto x^{2k} z + \beta$ with either $\beta \neq 0$ or $x^{2k} \neq 1$.
If $g$ fixes a transposition, there exists $i \neq j$ such that $g(\{i,j \}) = \{i,j \}$.
If $g(i) =i$ then $(x^2-1) i = \beta$ hence $i = \beta (x^2-1)^{-1}\beta$,
and similarly $j = \beta (x^2-1)^{-1}\beta = i$, contradicting $i \neq j$
(alternatively : this would contradict the fact that $G$ is a Frobenius group !).
If $g(i) = j$ and $g(j) = i$, then we get $(x^{2k}+1)(i-j) = 0$. Since $i-j \neq 0$
this implies $x^{2k} = -1$, hence $-1$ is a square inside $\F_p$.
This is not possible when $p \equiv 3 \mod 4$, and this concludes the proof.

\end{proof}

\subsection{Table for exceptional groups}

\label{sec:tableexc}

In table \ref{tab:exceptions} we gather the datas on the exceptional
complex reflection groups which are relevant for our studies.  The column $S$ indicates the 2-Sylow subgroup of $W/Z = W/Z(W)$.
Here $D_n$ is the dihedral group of order $n$, $Q_8$ is the quaternion
group of order $8$, and $\Z_n = \Z/n\Z$. In our description of $W/Z(W)$ we have taken care of
choosing the description where the 2-Sylow subgroup is the most apparent, when possible.
The gap from the family of all groups of a certain rank to the next one is realized by a
horizontal double line separating them.
The column `regular numbers' contains the regular numbers for $W$ but $1$ (which is regular
for every group), where
the freely regular numbers have been circled. Recall that the list of regular numbers
are characterized among all positive integer numbers by the fact that they divide exactly the same number of degrees and codegrees for $W$.
Therefore, this list can be deduced from the tables for degrees and codegrees, such as the one in
\cite{LEHRTAY}, table D.3, p. 275.

\begin{table}
$$
\begin{array}{|c|r|c|r|c|r|c|c|}
\hline
W & |\mathcal{A}| & W/Z & |W/Z| & |Z| & \mathrm{regular\ numbers} & \kappa(W) & S \\
\hline
\hline
G_{4} &4& \mathfrak{A}_4 & 12 & 2 &   2, 3, 4, 6   &6 & \Z_2^2 \\ 
\hline G_{5} &8& \mathfrak{A}_4& 12 & 6 &  2, 3, 4, 6, 12  &  6 & \Z_2^2\\
 \hline G_{6} &10 &\mathfrak{A}_4 & 12 & 4 &   2, 3, 4, 6, 12  & 12 &\Z_2^2\\ 
\hline G_{7} &14 &\mathfrak{A}_4 & 12 & 12 &   2, 3, 4, 6, 12  & 12 & \Z_2^2\\
 \hline G_{8} &6 &\mathfrak{S}_4  & 24 & 4 &  2, \circled{3}, 4, 6, 8, 12 & 4 &D_8 \\ 
 \hline G_{9} & 18  &\mathfrak{S}_4 & 24 & 8 &  2, \circled{3}, 4, 6, 8, 12, 24 &8 & D_8\\
  \hline G_{10} & 14 &\mathfrak{S}_4 & 24 & 12 &   2, 3, 4, 6, 8, 12, 24  &12 & D_8\\
   \hline G_{11} & 26 &\mathfrak{S}_4 & 24 & 24 &   2, 3, 4, 6, 8, 12, 24  & 24 & D_8\\ 
   \hline G_{12} &12 &\mathfrak{S}_4 & 24 & 2 &  2, \circled{3}, 4, 6, 8 & 2 & D_8\\ 
   \hline G_{13} & 18 &\mathfrak{S}_4 & 24 & 4 &   2, \circled{3}, 4, 6, 12 & 8& D_8\\
    \hline G_{14} &20 &\mathfrak{S}_4 & 24 & 6 &   2, 3, 4, 6, 8, 12, 24  & 6 & D_8\\
     \hline G_{15} &26 &\mathfrak{S}_4 & 24 & 12 &    2, 3, 4, 6, 12  &24 & D_8\\
      \hline G_{16} &12 &\mathfrak{A}_5  & 60 & 10 &  2, \circled{3}, 4, 5, 6, 10, 15, 20, 30 & 10 & \Z_2^2\\
  \hline G_{17} & 42 &\mathfrak{A}_5 & 60 & 20 &  2, \circled{3}, 4, 5, 6, 10, 12, 15, 20, 30, 60 & 20 & \Z_2^2\\ 
  \hline G_{18} & 32 &\mathfrak{A}_5 & 60 & 30 &   2, 3, 4, 5, 6, 10, 12, 15, 20, 30, 60  &30 & \Z_2^2\\
   \hline G_{19} & 62 & \mathfrak{A}_5& 60 & 60 &   2, 3, 4, 5, 6, 10, 12, 15, 20, 30, 60 & 60 & \Z_2^2\\ 
   \hline G_{20} & 20 &\mathfrak{A}_5 & 60 & 6 &  2, 3, 4, \circled{5}, 6, 10, 12, 15, 30  & 6 &  \Z_2^2\\
    \hline G_{21} & 50 &\mathfrak{A}_5 & 60 & 12 &  2, 3, 4, \circled{5}, 6, 10, 12, 15, 20, 30, 60 & 12 & \Z_2^2\\
     \hline G_{22} & 30 &\mathfrak{A}_5 & 60 & 4 &  2, \circled{3}, 4, \circled{5}, 6, 10, 12, 20  & 4 & \Z_2^2\\
      \hline \hline
G_{23} & 15 &\mathfrak{A}_5 & 60 & 2 &  2, \circled{3}, \circled{5}, 6, 10  & 2 & \Z_2^2\\ \hline
 G_{24} & 21 & \SL_3(\F_2) & 168 & 2 &  2, \circled{3}, 6, \circled{7}, 14& 2 & D_8 \\ \hline 
G_{25} & 12 & \SU_3(\F_2) & 216 & 3 &   2, 3, 4, 6, 9, 12   & 6 & Q_8 \\ \hline 
 G_{26} & 21 & \SU_3(\F_2) & 216 & 6 &    2, 3, 6, 9, 18  &6 & Q_8 \\ \hline 
 G_{27} & 45 &\mathfrak{A}_6 = \Sp_4(\F_2)' & 360 & 6 &  2, 3, \circled{5}, 6, 10, 15, 30& 6 & D_8 \\ \hline \hline 
 G_{28} & 24 &2^4 \rtimes (\mathfrak{S}_3)^2 & 576 & 2 &  2, \circled{3}, 4, 6, 8, 12 & 2 & \\ \hline 
 G_{29} & 40 & 2^4 \rtimes \mathfrak{S}_5 & 1920 
& 4 &  2, 4, \circled{5}, 10, 20 & 4 & \\ \hline 
G_{30} & 60 & \mathfrak{A}_5 \wr 2 & 7200 & 2 &   2, \circled{3}, 4, \circled{5}, 6, 10, 12, \circled{15}, 20, 30 & 2 & \\ \hline 
G_{31} & 60 & 2^4.\Sp_4(\F_2)& 11520 & 4 &  2, \circled{3}, 4, \circled{5}, 6, 8, 10, 12, 20, 24  & 4 & \\ \hline 
G_{32} & 40 & \SU_4(\F_2) & 25920 & 6 &  2, 3, 4, \circled{5}, 6, 8, 10, 12, 15, 24, 30  & 6 & \\ \hline \hline 
G_{33} & 45 & \SU_4(\F_2) & 25920 & 2 &  2, 3, \circled{5}, 6, 9, 10, 18&6 & \\ \hline \hline 
G_{34} & 126 & \PSU_4(\F_3) \rtimes \Z_2 & 6531840 & 6 &  2, 3, 6, \circled{7}, 14, 21, 42 & 6 & \\ \hline 
G_{35} & 36 & O_6^-(\F_2) & 51840 & 1 &  2, \circled{3}, 4, 6, 8, \circled{9}, 12& 2 &\\ \hline\hline
 G_{36} &63 & \SO_7(\F_2) & 1451520 & 2 & 2, \circled{3}, 6, \circled{7}, \circled{9}, 14, 18  & 2 & \\ \hline\hline 
 G_{37} & 120 &\SO_8^+(\F_2) & 348364800 & 2 &
  2, \circled{3}, 4, \circled{5}, 6, 8, 10&2 & \\
  & & & & &  12, \circled{15}, 20, 24, 30  & & \\ \hline
 \end{array}
 $$
 
 \caption{Table for exceptional groups.}
 \label{tab:exceptions}
 \end{table}

\section{K\"ahler manifolds}
\label{sect:kahler}

We notice in this section that our construction sometimes provides K\"ahler
manifolds. More precisely, letting $S$ denote the Sylow 2-subgroup of $W/Z(W)$,
it happens that the corresponding $N$-manifold can be endowed with a flat K\"ahler
structure which, by \cite{JOHNSONREES}, is equivalent to saying that
the corresponding Bieberbach group is a discrete cocompact torsion-free
subgroup of $U_{\frac{N}{2}} \ltimes \C^{N/2}$. As mentioned above, to
check this it is enough to compute the character table of the Sylow subgroup and
the character of its permutation action on $\mathcal{A}$. We get that this
phenomenon happens exactly, as far the exceptional groups are
concerned, for $G_5$, $G_6$, $G_9$, $G_{13}$, $G_{17}$, $G_{18}$, $G_{21}$, $G_{25}$, yielding K\"ahler manifolds of (real) dimension $8,10,18,18,42,32,50$ and $12$, respectively, with the holonomy group
given by the table.
Of course, by considering smaller 2-groups, one might
obtain K\"ahler manifolds from other groups as well. For instance, it can be checked that each
subgroup of order $2$ of $W/Z(W)$ for $W$ of types $G_4$, $G_7$, $G_{16}$ provides a K\"ahler manifold, of
complex dimensions $2$, $7$ and $6$, respectively. It is also the case in types $G_8$, $G_{12}$ for \emph{some} of the subgroups of order $2$.

Moreover, we checked that the flat manifold of real dimension $10$ and holonomy group
$\mathfrak{A}_4$ built in corollary \ref{cor:G4} from the whole complex reflection group $G_6$,
is a K\"ahler manifold. This is not the case for the one afforded by $G_4$, since it is not even
the case for the manifold afforded by its 2-Sylow subgroup, as noticed above.

We now describe, as an illustrative example, a kind of K\"ahler manifold arising systematically. Let $m \geq 3$ be an odd
integer and let $W$ be the dihedral group of order $4m$. Then $Z(W) = \Z_2$,
and $|\mathcal{A}| = 2m$. Let $\overline{W} = W/Z$. As $\Q \overline{W}$-modules,
$P_0 \otimes_{\Z} \Q \simeq P^{ab}\otimes_{\Z} \Q \simeq \Q \mathcal{A}$
and therefore the K\"ahler condition, for an arbitrary 2-subgroup $Q$ of $\overline{W}$, can be checked
on the $\Q Q$-module $\Q \mathcal{A}$. Let us set $Q =   \{ \bar{1}, \bar{s} \}$
where $s$ is a reflection in $W$ and $\bar{s}$ its class modulo $Z(W)$. There are
exactly 2 hyperplanes fixed by the action of $\bar{s}$ (the hyperplane fixed by $s$ and its orthogonal), and this is enough to ensure that the K\"ahler condition is satisfied. 
For a concrete description of the group, note that $B$ is an Artin group of Coxeter type $I_2(2m)$, with generators $\sigma, \tau$
and relation $(\sigma \tau)^m = (\tau \sigma)^m$. Our group can be constructed as the inverse image inside the group $B/P_0$ of the subgroup generated by a reflection. Therefore, one gets that
it admits a presentation with generators $\sigma, x_k$ with $k \in \Z/2m\Z$,
and relations $\sigma^2 = x_0$, $\sigma x_k \sigma^{-1} = x_{-k}$.


\bigskip
\bigskip
\bigskip


\begin{thebibliography}{00}
\bibitem{BECK} V. Beck, {\it Abelianization of subgroups of reflection groups and their braid groups: an application to cohomology}, Manuscripta Math. {\bf 136} (2011), 273--293.
\bibitem{BMR} M. Brou\'e, G. Malle, R. Rouquier,
{\it Complex reflection groups, braid groups, Hecke
algebras}, J. Reine Angew. Math. {\bf 500} (1998), 127--190.
\bibitem{DEKIMPE} K. Dekimpe, {\it Almost Bieberbach groups: affine and poynomial structures}, Springer LNM 1639 (1996).
\bibitem{DIGNE} F. Digne, {\it Pr\'esentation des groupes de tresses pures et de certaines de leurs
extensions}, preprint 1999, arXiv:1511.08731v1.
\bibitem{DMM} F. Digne, I. Marin, J. Michel, {\it The center of pure complex braid groups}, J. Algebra {\bf 347} (2011), 206--213.
\bibitem{GGO} D.L. Gon\c calves, J. Guaschi, O. Ocampo, {\it Quotients of the Artin braid
groups and crystallographic groups}, preprint 2015, arXiv:1503.04527v1.
\bibitem{HOFFMAN} M. Hoffman, {\it An invariant of finite abelian groups}, Amer. Math. Monthly {\bf 94} (1987), 664-666.
\bibitem{JOHNSONREES} F.E.A. Johnson, E.G. Rees, {\it K\"ahler groups and rigidity phenomena}, Math. Proc. Camb. Phil. Soc. {\bf 109} (1991), 31-44.
\bibitem{LEHRTAY} G.I. Lehrer, D.E. Taylor, {\it Unitary reflection groups}, Cambridge University Press, 2009.
\bibitem{ARREFL} I. Marin, {\it  Reflection groups acting on their hyperplanes}, J. Algebra {\bf 322} (2009), 2848–2860.
\bibitem{ORLIKTERAO} P. Orlik, H. Terao, {\it Arrangements of hyperplanes},  Springer-Verlag, Berlin, 1992.
\bibitem{RICHARDSON}  R.W. Richardson, {\it Conjugacy classes of involutions in Coxeter groups},
Bull. Austral. Math. Soc. {\bf 26} (1982), 1–15.
\bibitem{SPRINGER} T.A. Springer, {\it Regular elements of finite reflection groups}, Invent. Math. {\bf 25} (1974), 159--198.
\bibitem{TITS} J. Tits, {\it Normalisateurs de tores I : Groupes de Coxeter \'etendus}, J. Algebra {\bf 4} (1966), 96-116.
\end{thebibliography}
\end{document}